\documentclass[11pt, reqno]{amsart}

\usepackage{graphicx}

\usepackage{epic}
\usepackage{amsmath}
\usepackage{amsfonts}
\usepackage{amssymb}
\usepackage{amsthm}
\usepackage{eucal}
\usepackage[ansinew]{inputenc}

\newtheorem{theorem}{Theorem}[section]
\newtheorem{lemma}[theorem]{Lemma}

\newtheorem{proposition}[theorem]{Proposition}

\newtheorem{definition}[theorem]{Definition}
\newtheorem*{ack}{Acknowledgements}
\numberwithin{equation}{section}

\begin{document}

\title[Trigonometric Darboux transformations]
{Trigonometric Darboux transformations and Calogero-Moser matrices}

\author[L.~Haine]{Luc Haine}
\address{L.H., Department of Mathematics, Universit\'e catholique de
Louvain, Chemin du Cyclotron 2, 1348 Louvain-la-Neuve, Belgium}
\email{luc.haine@uclouvain.be}

\author[E.~Horozov]{Emil~Horozov}
\address{E.H., Department of Mathematics and Informatics,
Sofia University, 5 J. Bourchier Blvd., Sofia 1126, Bulgaria}
\address{Institute of Mathematics and Informatics, Bulg. Acad. of Sci.,
Acad. G. Bonchev Str., Block 8, 1113 Sofia, Bulgaria }
\email{horozov@fmi.uni-sofia.bg}

\author[P.~Iliev]{Plamen~Iliev}
\address{P.I., School of Mathematics, Georgia Institute of Technology,
Atlanta, GA 30332--0160, USA}
\email{iliev@math.gatech.edu}

\subjclass[2000]{35Q53, 37K10}
\keywords{KdV-like equations, Calogero-Moser spaces}
\date{January 30, 2008}
\begin{abstract} We characterize in terms of Darboux transformations the
spaces in the Segal-Wilson rational Grassmannian, which lead to commutative
rings of differential operators having coefficients which are rational
functions of $e^x$. The resulting subgrassmannian is parametrized in terms of
trigonometric Calogero-Moser matrices.
\end{abstract}

\maketitle
%%%%%%%%%%%%%%%%%%%%%%%%%%%%%%%%%%%%%%%%%%%%%%%%%%%%%%%%%%%%%%%%%%%%%%%%%%%%%%%%%%%%%%%%%%%%%%%%%%%%%%%% Section 1
\section{Introduction}
The rational Segal-Wilson Grassmannian $Gr^{rat}$ parametrizes the soliton
solutions of the Kadomtsev-Petviashvili (KP) equation. In \cite{W1}, Wilson
embarked upon a study of a subgrassmannian $Gr^{ad}\subset Gr^{rat}$, that he
called the adelic Grassmannian, which parametrizes the solutions of the KP
equation, rational in $x$ and vanishing as $x\to\infty$. The adelic Grassmannian
has (and is indeed characterized by) a remarkable bispectral involution
$V\to b(V), V\in Gr^{ad}$, which exchanges the role of the "space" and the
"spectral" variables in the corresponding stationary wave functions
$\psi_V(x,z)$, that is
\begin{equation}\label{1.1}
\psi_{b(V)}(x,z)=\psi_V(z,x).
\end{equation}
In particular, the reduced stationary wave function $e^{-xz}\psi_V(x,z)$
depends rationally not only on $z$, but also on $x$. In \cite{W2}, Wilson gave
an illuminating explanation for this involution, by showing that
\begin{equation} \label{1.2}
\psi_V(x,z)=e^{xz}\mbox{det}\big\{I-(xI-X)^{-1}(zI-Z)^{-1}\big\},
\end{equation}
with $(X,Z)$ an element of a so-called Calogero-Moser space
\begin{equation*}
C_N=\big\{(X,Z)\in gl(N,\mathbb{C})\times gl(N,\mathbb{C}):
\mbox{rank}([X,Z]+I)=1\big\}/GL(N,\mathbb{C}).
\end{equation*}
Here, $gl(N,\mathbb{C})$ denotes the space of complex $N\times N$ matrices,
$I$ is the identity matrix and the complex linear group $GL(N,\mathbb{C})$
acts by simultaneous conjugation of $X$ and $Z$. The bispectral involution
\eqref{1.1} becomes transparent when expressed at the level of Calogero-Moser
spaces, as it is given by $(X,Z)\to (Z^t,X^t)$, where $X^t$ and $Z^t$ are the
transposes of $X$ and $Z$.

In \cite{Ha}, one of us, motivated by previous studies of the first and third
authors on a discrete-continuous version of the bispectral problem in
\cite{HI} and \cite{I}, suggested to study yet another subgrassmannian of
$Gr^{rat}$, the so-called trigonometric Grassmannian
$Gr^{trig}\subset Gr^{rat}$, characterized by the property that the reduced
stationary wave function $e^{-xz}\psi_V(x,z)$ should depend rationally on
$e^x$. The goal of this note is to establish a formula similar to \eqref{1.2},
for a space $V\in Gr^{trig}$, namely
\begin{equation}\label{1.3}
\psi_V(x,z)=e^{xz}\mbox{det}\big\{I-X(e^xI-X)^{-1}(zI-Z)^{-1}\big\},
\end{equation}
with $(X,Z)$ belonging now to a trigonometric Calogero-Moser space
\begin{multline}\label{1.4}
C_N^{trig}=\big\{(X,Z)\in GL(N,\mathbb{C})\times gl(N,\mathbb{C}): \\
\mbox{rank}(XZX^{-1}-Z+I)=1\big\}/GL(N,\mathbb{C}).
\end{multline}
The formula was conjectured in \cite{Ha} and established only
for a generic situation. There are two
main ingredients in the proof, which will allow us to derive the result from
\eqref{1.2}. First, the notion of bispectral Darboux transformations as
introduced and used in other contexts by the second author in collaboration
with Bakalov and Yakimov in \cite{BHY1} and \cite{BHY2}; second, the discovery
in \cite{HI} that if $\tau(t_1,t_2,t_3,\ldots)$ is a tau-function of the KP
hierarchy, then
\begin{equation}\label{1.5}
\tau(n,t_1,t_2,t_3,\ldots)=\tau\big(t_1+n,t_2-\frac{n}{2},t_3+\frac{n}{3},
\ldots\big),\quad n\in\mathbb{Z},
\end{equation}
is a tau-function of the discrete KP hierarchy. Indeed, the trick to
establish \eqref{1.3} is to show via the technique of bispectral Darboux
transformations, that
\begin{equation}
\psi_V^b(n,z)=\psi_V(\mbox{log}(1+z),n),
\end{equation}
is a (stationary) wave function of the discrete KP hierarchy, which can be
constructed from an adelic tau-function $\tau_{b(V)}(t_1,t_2,\ldots)$, with
$b(V)\in Gr^{ad}$, via the formula \eqref{1.5}.

Finally, we like to mention that during the ISLAND 3 conference where this
work was presented, Oleg Chalykh and Alexander Varchenko informed us that they
have obtained results related to ours, though with rather different aims and
techniques, see \cite{BC} and \cite{MTV}.
%%%%%%%%%%%%%%%%%%%%%%%%%%%%%%%%%%%%%%%%%%%%%%%%%%%%%%%%%%%%%%%%%%%%%%%%%%%%%%%%%%%%%%%%%%%%%%%%%%%%%%%%%% Section 2
\section{The rational Grassmannian}
In this section, we review the definition of the rational Grassmannian
$Gr^{rat}$ in terms of Darboux transformations, following \cite{BHY1}.
\begin{definition}
A function $\psi(x,z)$ is a Darboux transform of $e^{xz}$, if and only if
there exist monic polynomials $f(z),g(z)$ and monic differential operators
$P(x,\partial), Q(x,\partial)$, with $\partial=\frac{\partial}{\partial x}$,
such that
\begin{align}
\psi(x,z)=\frac{1}{f(z)}P(x,\partial)e^{xz}\label{2.1}\\
e^{xz}=\frac{1}{g(z)}Q(x,\partial)\psi(x,z),\label{2.2}
\end{align}
with the order of $P(x,\partial)$ equal to the degree of $f(z)$.
\end{definition}

Obviously
\begin{equation*}
Q(x,\partial)P(x,\partial)e^{xz}=f(z)g(z)e^{xz},
\end{equation*}
so that denoting the polynomial $f(z)g(z)$ by $h(z)$, we see that
\begin{equation}\label{2.3}
h(\partial)=Q(x,\partial)P(x,\partial).
\end{equation}
On the other hand, $\psi(x,z)$ satisfies
\begin{equation}\label{2.4}
P(x,\partial)Q(x,\partial)\psi(x,z)=f(z)g(z)\psi(x,z),
\end{equation}
showing that the operator $L=PQ$ is a traditional Darboux transformation of
the operator $h(\partial)$ with constant coefficients, which justifies the
terminology.

We shall denote by $\mathbb{C}[z]$ (resp. $\mathbb{C}(z)$) the space of
polynomials (resp. rational functions) in $z$. According to \cite{W1}, the
rational Grassmannian $Gr^{rat}$ is formed by subspaces
$V\subset \mathbb{C}(z)$ satisfying
\begin{equation}\label{2.5}
g(z)\mathbb{C}[z]\subset V\subset f(z)^{-1}\mathbb{C}[z],
\end{equation}
where $f(z)$ and $g(z)$ are polynomials, and the codimension of $V$ in
$f(z)^{-1}\mathbb{C}[z]$ is equal to the degree of $f(z)$. Using the
terminology of Sato's theory of the KP equation, the stationary wave function
$\psi_V(x,z)$ (the Baker function in \cite{W1}) of such a space turns out to
be a Darboux transform of $e^{xz}$, as defined above. In fact, the two notions
are equivalent. We shall need the bilinear form on the space $\mathbb{C}(z)$
defined by
\begin{equation}\label{2.6}
B(u,v)=\mbox{res}_zu(z)v(z),\quad u,v\in\mathbb{C}(z),
\end{equation}
with $\mbox{res}_z$ the coefficient of $z^{-1}$ in the Laurent expansion
around $\infty$.
\begin{proposition} A function $\psi(x,z)$ is the stationary wave function of
a space $V\in Gr^{rat}$ if and only if it is a Darboux transform of $e^{xz}$.
\end{proposition}
\begin{proof} We only sketch the proof of the "if" part. The operator $P$ in
\eqref{2.1} being monic, it is given by
\begin{equation}\label{2.7}
P(\phi)
=\frac{\mbox{Wr}(\phi_1,\ldots,\phi_K,\phi)}{\mbox{Wr}(\phi_1,\ldots,\phi_K)},
\end{equation}
with $\phi_1,\ldots,\phi_K$, a basis of the kernel of $P$, where $Wr$ denotes
the Wronski determinant. From the factorization \eqref{2.3}, we have that
$\mbox{ker}\;P\subset \mbox{ker}\;h(\partial)$ and
$\mbox{ker}\;Q=P(\mbox{ker}\;h(\partial))$. Since $h(\partial)$ is a constant
coefficients operator, it follows that the coefficients of $P$ and $Q$ are
rational functions of $x$ and $e^{\lambda x}$, for a finite number of values
of $\lambda$.

Hence, if we introduce the function
\begin{equation}\label{2.8}
\psi^*(x,z)=\frac{1}{g(z)}Q^*e^{-xz},
\end{equation}
with $Q^*$ the formal adjoint of $Q$, defined by
$(a(x)\partial)^*=-\partial\circ a(x)$, both
$e^{-xz}\psi(x,z)$ and $e^{xz}\psi^*(x,z)$ are rational functions of
$x,e^{\lambda x}$ and $z$. Assuming (without loss of generality) that these
functions are regular at $x=0$, the coefficients of the Taylor expansions of
$\psi(x,z)$ and $\psi^*(x,z)$ around $x=0$, generate two subspaces $V$ and
$V^*$ of $\mathbb{C}(z)$, such that
\begin{align}
V&=\mbox{span} \{\partial^i\psi(x,z)_{\vert x=0},\;
i=0,1,2,\ldots\}\subset f(z)^{-1}\mathbb{C}[z],\label{2.9}\\
V^*&=\mbox{span}\{\partial^i\psi^*(x,z)_{\vert x=0},\;i=0,1,2,\ldots\}
\subset g(z)^{-1}\mathbb{C}[z],\label{2.10}
\end{align}
and the codimension of $V$ in $f(z)^{-1}\mathbb{C}[z]$ is equal to the degree
of $f$.

The spaces $V$ and $V^*$ are orthogonal with respect to the bilinear form $B$
defined in \eqref{2.6}, as follows from a simple computation in the ring of
formal pseudo-differential operators. Indeed from \eqref{2.1} and \eqref{2.8},
\begin{equation}
\psi(x,z)=P\circ f(\partial)^{-1}e^{xz}\quad\mbox{and}\quad
\psi^*(x,z)=\big(g(\partial)^{-1}\circ Q\big)^*e^{-xz}.
\end{equation}
Hence, for all $i,j\geq 0$, by the "very simple and extremely useful
lemma" 6.2.5 in \cite{D}, using \eqref{2.3}, we compute
\begin{align*}
&\mbox{res}_z\partial^i\psi(x,z)\partial^j\psi^*(x,z)\\&=
(-1)^j\;\mbox{res}_z\Big(\partial^i\circ P\circ
f(\partial)^{-1}e^{xz}\Big)\Big(\big(g(\partial)^{-1}\circ
Q\circ\partial^j\big)^*e^{-xz}\Big)\\
&=(-1)^j\;\mbox{res}_{\partial}\big(\partial^i\circ P\circ
f(\partial)^{-1}\big)\circ \big(g(\partial)^{-1}\circ Q\circ\partial^j\big)\\
&=(-1)^j\;\mbox{res}_{\partial}\partial^i\circ P\circ (Q\circ P)^{-1}\circ
Q\circ\partial^j\\
&=(-1)^j\;\mbox{res}_{\partial}\partial^{i+j}=0,
\end{align*}
with $\mbox{res}_\partial\sum a_k\partial^k=a_{-1}$. Taking the orthogonal
complement with respect to $B$ of \eqref{2.10}, we deduce that
$g(z)\mathbb{C}[z]\subset V$, which combined with \eqref{2.9}, establishes
that $V\in Gr^{rat}$, as defined in \eqref{2.5}. It is easy to check that
$\psi(x,z)=\psi_V(x,z)$, which finishes the proof.
\end{proof}
%%%%%%%%%%%%%%%%%%%%%%%%%%%%%%%%%%%%%%%%%%%%%%%%%%%%%%%%%%%%%%%%%%%%%%%%%%%%%%%%%%%%%%%%%%%%%%%%%%%%%%%% Section 3
\section{The trigonometric Grassmannian}
\begin{definition} A Darboux transform $\psi(x,z)$ of $e^{xz}$, will be called
trigonometric if and only if the operators $P$ and $Q$ in \eqref{2.1} and
\eqref{2.2}, have coefficients which are rational functions of $e^x$, i.e.
$P,Q\in\mathbb{C}(e^x)[\partial]$.
\end{definition}
In this section, we characterize the trigonometric Darboux transforms. With
the notations of the previous section, let us write the constant coefficients
operator $h(\partial)$ in \eqref{2.3} as
\begin{equation} \label{3.1}
h(\partial)=\prod_{r=1}^n\prod_{j=0}^{n_r}(\partial-\lambda_r+j)^{m_{r,j}},
\end{equation}
where $\lambda_1,\ldots,\lambda_n$ are distinct complex numbers such that
$\lambda_r-\lambda_s\notin\mathbb{Z}$ for $r\neq s$ and $m_{r,j}$, the
multiplicities of the roots $\lambda_{r}-j$, are nonnegative integers, with
$m_{r,0}>0$. Then, the kernel of $h(\partial)$ is given by
\begin{equation*}
\mbox{ker}\;h(\partial)=\bigoplus_{r=1}^n W_r,
\end{equation*}
where
\begin{equation*}
W_r=\mbox{span}\{x^k e^{(\lambda_r-j)x}:k=0,1,\ldots,m_{r,j}-1,
\;j=0,1,\ldots,n_r\}.
\end{equation*}
\begin{lemma} Let $P\in \mathbb{C}(e^x)[\partial]$ be an operator such that
the factorization \eqref{2.3} holds. Then
\begin{align*}
i)\;&\phi(x)\in \emph{ker}\;P\Rightarrow\phi(x+2\pi il)\in\emph{ker}\;P,\;
\forall\;l\in\mathbb{Z}.\\
ii)\;&\emph{ker}\;P=\bigoplus_{r=1}^n \big(W_r\cap \emph{ker}\;P\big).
\end{align*}
\end{lemma}
\begin{proof} i) The assertion follows immediately from the invariance of the
coefficients of $P$ under the change $x\to x+2\pi il, l\in\mathbb{Z}$, since
they are rational functions of $e^x$.

ii) Since $\mbox{ker}\;P\subset \mbox{ker}\;h(\partial)$, any $\phi(x)$ in
the kernel of $P$ can be expanded as
\begin{equation*}
\phi(x)=\sum_{r=1}^n\Big\{\sum_{k=0}^{k_r} p_{r,k}(e^{-x})x^k\Big\}
e^{\lambda_rx},
\end{equation*}
with $p_{r,k}(e^{-x})\in\mathbb{C}[e^{-x}]$, some polynomial in $e^{-x}$.
The result will follow by induction, if we can show that for every
$r\in\{1,2,\ldots,n\}$, there exits an element
$\hat{\phi}_r(x)\in\mbox{ker}\;P$ of the form
\begin{equation}\label{3.2}
\hat{\phi}_r(x)=\Big\{p_{r,k_r}(e^{-x})x^{k_r}
+\sum_{k<k_r}\hat{p}_{r,k}(e^{-x})x^k\Big\}e^{\lambda_r x},
\end{equation}
with the same polynomial $p_{r,k_r}(e^{-x})$ in front of the highest power
$x^{k_r}$, and some other polynomials $\hat{p}_{r,k}(e^{-x})$ as coefficients
of $x^k,k<k_r$.

To establish the existence of $\hat{\phi}_r(x)$, we observe that since
$\phi(x+2\pi i)\in \mbox{ker}\;P$, for any $s\in\{1,2,\ldots,n\}$, we have that
\begin{align*}
\overline{\phi}(x)&=\phi(x)-e^{-2\pi i\lambda_s}\phi(x+2\pi i)\\
&=\sum_{r=1}^n\Big\{\big(1-e^{2\pi i(\lambda_r-\lambda_s)}\big)
p_{r,k_r}(e^{-x})x^{k_r}+
\sum_{k<k_r}\overline{p}_{r,k}(e^{-x})x^k\Big\}e^{\lambda_rx},
\end{align*}
is an element of the kernel of $P$. Since
$\lambda_r-\lambda_s\not\in\mathbb{Z}$, for $r\neq s$, the coefficient of
$x^{k_r}$ is the same polynomial $p_{r,k_r}(e^{-x})$ multiplied by a nonzero
constant for $r\neq s$, and vanishes for $r=s$. Iterating the process, we can
produce an element from the kernel of the form
\begin{equation*}
\sum_{r\neq s}\Big\{c_{r,k_r}p_{r,k_r}(e^{-x})x^{k_r}+
\sum_{k<k_r}\overline{p}_{r,k}(e^{-x})x^k\Big\}e^{\lambda_rx},
\end{equation*}
with nonzero constants $c_{r,k_r}$. Pursuing the process, we can eliminate all
exponentials $e^{\lambda_s x}$, except one, producing an element from the
kernel of the form \eqref{3.2}. This concludes the proof of the lemma.
\end{proof}

\begin{theorem} The Darboux transform defined by \eqref{2.1} and \eqref{2.2}
is trigonometric if and only if
\begin{equation*}
\emph{ker}\;P=\bigoplus_{r=1}^nK_r,
\end{equation*}
with $K_r$ a subspace of $W_r$ having a basis which is a union of sets of the
form
\begin{equation}\label{3.3}
\frac{1}{l!}\partial_y^l\Bigg(\sum_{j=0}^{n_r}\sum_{k=0}^{m_{r,j}-1}c_{r,k,j}
y^ke^{(\lambda_r-j)x}\Bigg)\Bigg\vert_{y=x},\quad l=0,1,\ldots,l_0,
\end{equation}
where $l_0=\emph{max}\{k: c_{r,k,j}\neq 0\;\emph{for some}\;j\}$.
\end{theorem}
\begin{proof}\footnote{The arguments used in the proof are analogous to
Lemmas 2.8 and 2.9 in \cite{BHY2}, see also \cite{CN} where conditions similar
to \eqref{3.3} can be found, without proof.} Let us assume that the Darboux
transform is trigonometric. From Lemma 3.2, we can choose a basis
$\{\phi_1,\ldots,\phi_K\}$ of the kernel of $P$, such that each $\phi_i$
belongs to some $W_{r_i}, r_i\in\{1,\ldots,n\}$. Any $\phi\in W_r$ can be
expanded as
\begin{equation*}
\phi(x)=\sum_{j=0}^{n_r}\sum_{k=0}^{m_{r,j}-1}c_{r,k,j}x^ke^{(\lambda_r-j)x}.
\end{equation*}
Since the coefficients of $P$ are rational functions of $e^x$, they are
$2\pi i$ periodic, hence, for every $l\in\mathbb{Z}$, we have
\begin{equation} \label{3.4}
e^{-2\pi il\lambda_r}\phi(x+2\pi il)=\sum_{j=0}^{n_r}\sum_{k=0}^{m_{r,j}-1}
c_{r,k,j}(x+2\pi il)^ke^{(\lambda_r-j)x}\;\in\mbox{ker}\;P.
\end{equation}
Obviously, \eqref{3.4} must also hold for every $l\in\mathbb{C}$.
Differentiating repeatitively this identity with respect to $l$ and putting
$l=0$, shows that the functions defined in \eqref{3.3} belong to the kernel of
$P$ too.

Conversely, if we can choose a basis $\phi_1,\ldots,\phi_K$ of the kernel of
$P$ with $\phi_i\in W_{r_i}$, then from the $i^{th}$ column of determinants in
the numerator and the denominator of \eqref{2.7}, we can factor off
$e^{\lambda_{r_i}x}$, which shows that the coefficients of the operator $P$
depend rationally on $x$ and $e^{x}$, i.e.
$P\in\mathbb{C}(x,e^x)[\partial]$. From \eqref{3.3}, it follows that
$\phi_1(x+2\pi i),\ldots,\phi_K(x+2\pi i)$ is again a basis of the kernel of
$P$. This shows that the coefficients of $P$ are $2\pi i$ periodic, hence,
they must be purely rational functions of $e^x$. The coefficients of the
operator $Q$ such that $QP=h(\partial)$ are then  automatically rational
functions of $e^x$ too, completing the proof.
\end{proof}

To specify completely the trigonometric Darboux transformation given a
factorization of \eqref{3.1} as in \eqref{2.3}, with
$P,Q\in\mathbb{C}(e^x)[\partial]$, we still need to determine the polynomials
$f(z)$ and $g(z)$ in \eqref{2.1} and \eqref{2.2}, for which there is some
arbitrariness. We fix the arbitrariness by observing that there is a unique
choice for the polynomial $f(z)$ such that
\begin{equation}\label{3.5}
\lim_{e^{x}\to\infty} \psi(x,z)e^{-xz}=1.
\end{equation}
By an argument similar to the one given in \cite{W1} (see Lemma 6.1 in
\cite{W1}, where this condition is imposed when $\psi(x,z)$ is a rational
function of $x$ instead of $e^x$), one can show that by normalizing the basis
$\phi_1,\ldots,\phi_K$ of the kernel of $P$, as described in Theorem 3.3, so
that $\phi_i\in W_{r_i}$ and
\begin{equation*}
\phi_i(x)=x^{k_i}e^{(\lambda_{r_i}-j_i)x}+(\mbox{terms involving only}\;x^k
e^{(\lambda_{r_i}-j)x}\;\mbox{with}\;j>j_i),
\end{equation*}
one must pick
\begin{equation}\label{3.6}
f(z)=\prod_{i=1}^K(z-\lambda_{r_i}+j_i).
\end{equation}
\begin{definition} The trigonometric Grassmannian $Gr^{trig}\subset Gr^{rat}$
is defined to be the set of spaces $V \in Gr^{rat}$ whose stationary wave
function $\psi_V(x,z)$ is obtained by a trigonometric Darboux transform of
$e^{xz}$, with the normalization of $f(z)$ specified as in \eqref{3.6}.
\end{definition}
%%%%%%%%%%%%%%%%%%%%%%%%%%%%%%%%%%%%%%%%%%%%%%%%%%%%%%%%%%%%%%%%%%%%%%%%%%%%%%%%%%%%%%%%%%%%%%%%%%%%%%%%% Section 4
\section{Bispectrality and Calogero-Moser matrices}
Let $T$ be the shift operator acting on functions of $z$ by
\begin{equation*}
Tf(z)=f(z+1).
\end{equation*}
The pair of equations
\begin{align*}
\partial e^{xz}&=ze^{xz}\\
Te^{xz}&=e^xe^{xz},
\end{align*}
defines an anti-isomorphism between the algebra of differential operators
whose coefficients are polynomials  in $e^{x}$ and the algebra of (positive)
difference operators whose coefficients are polynomials in $z$
\begin{equation*}
b:\mathbb{C}[e^x][\partial]\to \mathbb{C}[z][T],
\end{equation*}
with
\begin{equation}\label{4.1}
b(e^x)=T\quad \mbox{and}\quad b(\partial)=z.
\end{equation}

If $\psi(x,z)$ is a trigonometric Darboux transform of $e^{xz}$, we can write
the operators $P$ and $Q$ in \eqref{2.1} and \eqref{2.2} as
\begin{equation}\label{4.2}
P=\frac{1}{\theta(e^x)}\overline{P}\quad \mbox{and}\quad
Q=\overline{Q}\frac{1}{\nu(e^x)}\quad\mbox{with}\;
\overline{P},\overline{Q}\in\mathbb{C}[e^x][\partial],
\end{equation}
and $\theta(e^x), \nu(e^{x})$ some polynomials in $e^x$.
\begin{proposition} Let $\psi(x,z)$ be a trigonometric Darboux transform of
$e^{xz}$. Then
\begin{align}
\psi(x,z)&=\frac{1}{\theta(e^x)}\frac{1}{f(z)}b(\overline{P})e^{xz}
                                              \label{4.3}\\
e^{xz}&=\frac{1}{\nu(e^x)}b(\overline{Q})\frac{1}{g(z)}\psi(x,z),\label{4.4}
\end{align}
with $\overline{P},\overline{Q}$ as in \eqref{4.2}. As a consequence,
$\psi(x,z)$ in addition to be an eigenfunction of a differential operator in
$x$ as in \eqref{2.4}, is also an eigenfunction of a difference operator in $z$
\begin{equation*}
f(z)^{-1}b(\overline{P})b(\overline{Q})g(z)^{-1}\psi(x,z)=
\theta(e^x)\nu(e^x)\psi(x,z),
\end{equation*}
i.e. $\psi(x,z)$ solves a differential-difference bispectral problem.
\end{proposition}
\begin{proof} Equation \eqref{4.3} follows immediately from \eqref{2.1} and
\eqref{4.2}, using the definition of the bispectral map \eqref{4.1}. From
\eqref{2.3} and \eqref{4.2}, we have
\begin{equation*}
\overline{Q}\nu(e^x)^{-1}\theta(e^x)^{-1}\overline{P}=f(\partial)g(\partial),
\end{equation*}
implying
\begin{equation*}
b(\overline{P})\theta(T)^{-1}\nu(T)^{-1}b(\overline{Q})=f(z)g(z).
\end{equation*}
Hence
\begin{equation*}
\psi(x,z)=f(z)^{-1}b(\overline{P})\theta(T)^{-1}e^{xz}
=g(z)\big(\nu(T)^{-1}b(\overline{Q})\big)^{-1}e^{xz},
\end{equation*}
or, equivalently,
\begin{equation*}
e^{xz}=\nu(T)^{-1}b(\overline{Q})g(z)^{-1}\psi(x,z)\Leftrightarrow e^{xz}
=\nu(e^x)^{-1}b(\overline{Q})g(z)^{-1}\psi(x,z),
\end{equation*}
which establishes \eqref{4.4} and concludes the proof.
\end{proof}

Let us now assume that the trigonometric Darboux transformation has been
normalized as explained in \eqref{3.6}, or equivalently
$\psi(x,z)=\psi_V(x,z)$, for $V\in Gr^{trig}$, according to Definition 3.4.
Let us define
\begin{equation}\label{4.5}
\psi^b(n,z)=\psi_V(\mbox{log}(1+z),n).
\end{equation}
Because of the normalization, we deduce from \eqref{3.5} that
\begin{equation}\label{4.6}
\lim_{z\to\infty}\psi^b(n,z)(1+z)^{-n}=1.
\end{equation}
Moreover, putting
\begin{equation*}
\Delta=T-I,
\end{equation*}
after substituting $n$ for $z$ and $\mbox{log}(1+z)$ for $x$ in \eqref{4.3}
and \eqref{4.4}, it follows that
\begin{align}
\psi^b(n,z)&=\frac{1}{\theta(1+z)}R(n,\Delta)(1+z)^n\label{4.7}\\
(1+z)^n&=\frac{1}{\nu(1+z)}S(n,\Delta)\psi^b(n,z),\label{4.8}
\end{align}
with $R(n,\Delta)$ and $S(n,\Delta)$ some (positive) difference operators in
$\Delta$ (acting on functions depending on $n$), whose coefficients are
rational functions of $n$. We introduce the following definition.
\begin{definition} i) A function  $\psi(n,z)$ which satisfies \eqref{4.7} and
\eqref{4.8}, for some monic difference operators $R(n,\Delta), S(n,\Delta)$,
with $\theta(z)$ and $\nu(z)$ monic polynomials in $z$ such that the order of
$R$ is equal to the degree of $\theta$, will be called a discrete Darboux
transform of $(1+z)^n$.

ii) A discrete Darboux transform $\psi(n,z)$ of $(1+z)^n$, will be called
polynomial when the coefficients of $R(n,\Delta)$ and $S(n,\Delta)$ are
rational functions of $n$.
\end{definition}

From Wilson's result \cite{W2} for a space $\tilde{V}\in Gr^{ad}$, there
exists a Calogero-Moser pair $(\tilde{X},\tilde{Z})\in C_N$ such that
the corresponding tau-function is given by
\begin{equation}\label{4.9}
\tau_{\tilde{V}}(t_1,t_2,t_3,\ldots)
=\mbox{det}\Big\{\tilde{X}-\sum_{k=1}^{\infty}kt_k\tilde{Z}^{k-1}\Big\}.
\end{equation}
From this formula and \eqref{1.5} it follows that
\begin{equation*}
\tau_{\tilde{V}}(n,t_1,t_2,\dots)=
\tau_{\tilde{V}}\Big(t_1+n,t_2-\frac{n}{2},t_3+\frac{n}{3},\ldots\Big),
\end{equation*}
is a tau-function of the discrete KP hierarchy
\begin{equation*}
\frac{\partial L}{\partial t_i}=[(L^i)_+,L],\quad
L=\Delta+\sum_{j=0}^{\infty}a_j(n,t_1, t_2,\ldots)\Delta^{-j},
\end{equation*}
with $(L^i)_{+}$ the (positive) difference part of $L^i$.
The corresponding wave function (of the discrete KP hierarchy) is
\begin{multline} \label{4.10}
\psi_{\tilde{V}}(n,t,z)
=(1+z)^n\exp\Big(\sum_{k=1}^{\infty} t_k z^k\Big)\times\\
\frac{\tau_{\tilde{V}}\big(n,t_1-\frac{1}{z},t_2-\frac{1}{2z^2},
t_3-\frac{1}{3z^3},\ldots\big)}
{\tau_{\tilde{V}}\big(n,t_1,t_2,t_3,\ldots\big)}.
\end{multline}

A simple computation using \eqref{4.9} shows that
\begin{equation}\label{4.11}
\tau_{\tilde{V}}(n,t_1,t_2,\dots)=
\mbox{det}\Big\{\tilde{X}
-\sum_{k=1}^{\infty}kt_k\tilde{Z}^{k-1}-n(I+\tilde{Z})^{-1}\Big\}.
\end{equation}
In \cite{HI} it was assumed that the eigenvalues of $\tilde{Z}$ are inside
the unit circle, but clearly the right-hand side of \eqref{4.11} is
well defined as long as $-1$ is not an eigenvalue of $\tilde{Z}$.
Thus, analytic continuation shows that the above formula can be applied
when $\det(I+\tilde{Z})\neq0$.

If we denote by $\psi_{\tilde{V}}(n,z)=\psi_{\tilde{V}}(n,0,z)$, the
corresponding stationary wave function, then from \eqref{4.10} and
\eqref{4.11} it follows immediately that
\begin{equation}\label{4.12}
\lim_{n\rightarrow\infty}\psi_{\tilde{V}}(n,z)(1+z)^{-n}=1.
\end{equation}
It was shown in \cite{HI} that $\psi_{\tilde{V}}(n,z)$ is a polynomial
discrete Darboux transform of $(1+z)^n$. In fact, any polynomial discrete
Darboux transform of $(1+z)^n$ is obtained by the construction above, up to
a normalization, which can be fixed by imposing \eqref{4.12}. The result is
summarized in the next theorem, of which we sketch the idea of the proof.

\begin{theorem}
A function $\psi(n,z)$ is a polynomial discrete Darboux transform of
$(1+z)^n$ if and only if
$$\psi(n,z)=\frac{\theta_1(z)}{\theta_2(z)}\psi_{\tilde{V}}(n,z),$$
where $\theta_1(z)$ and  $\theta_2(z)$ are monic polynomials of the same
degree, and $\psi_{\tilde{V}}(n,z)$ is a stationary wave function of the
discrete KP hierarchy as in \eqref{4.10}-\eqref{4.11}, built from a space
$\tilde{V}\in Gr^{ad}$ corresponding to a Calogero-Moser pair
$(\tilde{X},\tilde{Z})$ such that $\det(I+\tilde{Z})\neq 0$.
\end{theorem}
The proof of the above theorem can be briefly explained as follows. First we
show that, up to a factor independent of $n$, polynomial discrete Darboux
transforms of $(1+z)^n$ can be characterized by the fact that the kernel of
the operator $R$ in \eqref{4.7}
has a basis consisting of functions of the form
$$\phi_j(n)=p_j(n)(\lambda_j+1)^n, \text{ where }
\lambda_j\in\mathbb{C}\setminus\{-1\},$$
and $p_j(n)$ are polynomials of $n$. The space $\tilde{V}\in Gr^{ad}$
corresponds to a Darboux transform of $e^{xz}$ (in the sense of
Definition~2.1) such that the polynomial $P(x,\partial)$ in \eqref{2.1} has a
kernel spanned by the functions
$$f_j(x)=p_j\big((1+z)\partial_z\big) e^{xz}|_{z=\lambda_j}.$$
The condition $\det(I+\tilde{Z})\neq 0$ in the theorem reflects the fact that
$\lambda_j\neq -1$.\\

In this note, we just like to explain how to deduce from this result a
parametrization of the trigonometric Grassmannian $Gr^{trig}$ in terms of
trigonometric Calogero-Moser matrices as defined in \eqref{1.4}.
We recall that Gekhtman and Kasman (see \cite{KG}, Corollary 3.2) have
established that for any triple $(X,Y,Z)$ of $N\times N$ matrices such that
rank $(XZ-YX)=1$, the function
\begin{equation*}
\tau_{(X,Y,Z)}(t_1,t_2,\ldots)=\mbox{det}
\Big\{I-X\exp\big\{-\sum_{k=1}^{\infty}t_k Z^k\big\}
\exp\big\{\sum_{k=1}^{\infty}t_k Y^k\big\}\Big\},
\end{equation*}
is a tau-function of the KP hierarchy, corresponding to some space of
$Gr^{rat}$. The next theorem shows that the special choice $Y=Z-I$ in their
formula, with $X$ invertible, characterizes tau-functions of spaces
of $Gr^{trig}$.
\begin{theorem} There is a one-to-one correspondence between trigonometric
Calogero-Moser pairs $(X,Z)$ (modulo conjugation) as defined in \eqref{1.4},
and tau-functions of spaces $V\in Gr^{trig}$, which is given by
\begin{equation}\label{4.13}
\tau_V(t_1,t_2,\ldots)=\emph{det} \Big\{I-X\emph{exp}
\Big\{\sum_{k=1}^{\infty}t_k\big((Z-I)^k-Z^k\big)\Big\}\Big\}.
\end{equation}
\end{theorem}
\begin{proof} Let $\psi_V(x,z)$ be the stationary wave function of a space
$V\in Gr^{trig}$. From \eqref{4.6}, \eqref{4.7} and \eqref{4.8}, it follows
that $\psi^b(n,z)$ as defined in \eqref{4.5} is a polynomial Darboux transform
of $(1+z)^n$ such that
$$\lim_{n\rightarrow\infty}\psi^b(n,z)(1+z)^{-n}=1.$$
From Theorem 4.3 and \eqref{4.12}, there exists a space $\tilde{V}\in Gr^{ad}$
and a pair of matrices $(\tilde{X},\tilde{Z})$
such that $\psi^b(n,z)=\psi_{\tilde{V}}(n,z)$ can be computed via formulae
\eqref{4.10}-\eqref{4.11}.
By an easy computation we obtain
\begin{equation*}
\psi^b(n,z)=(1+z)^n\mbox{det}\Big\{I+\big(\tilde{X}
-n(I+\tilde{Z})^{-1}\big)^{-1}\big(zI-\tilde{Z}\big)^{-1}\Big\}.
\end{equation*}

Hence, by the definition of $\psi^b(n,z)$ in \eqref{4.5},
\begin{align*}
\psi_V(x,z)&=\psi^b(z,e^x-1)\nonumber\\
&=e^{xz}\mbox{det}\Big\{I+\big(\tilde{X}-z(I+\tilde{Z})^{-1}\big)^{-1}
\big((e^x-1)I-\tilde{Z}\big)^{-1}\Big\}.
\end{align*}
Defining
\begin{equation}\label{4.14}
X=I+\tilde{Z}^t,\quad Z=\tilde{X}^t(I+\tilde{Z}^t),
\end{equation}
since a determinant is invariant by transposition, we obtain
\begin{align}
\psi_V(x,z)&=e^{xz}\mbox{det}\Big\{I+\big(e^xI-(I+\tilde{Z}^t)\big)^{-1}
\big(\tilde{X}^t-z(I+\tilde{Z}^t)^{-1}\big)^{-1}\Big\}\nonumber\\
&=e^{xz}\mbox{det}\Big\{I-X(e^xI-X)^{-1}(zI-Z)^{-1}\Big\}\label{4.15},
\end{align}
where in the last equation, we have used that the multiplication of the
matrices $X$ and $(e^xI-X)^{-1}$ commutes.
Since $(\tilde{X},\tilde{Z})$ is a Calogero-Moser pair, we deduce that
\begin{multline*}
\mbox{rank}\;(XZX^{-1}-Z+I)=\mbox{rank}\;([X,ZX^{-1}]+I)\\=\mbox{rank}\;
([\tilde{Z}^t,\tilde{X}^t]+I)
=\mbox{rank}\;([\tilde{X},\tilde{Z}]+I)=1,
\end{multline*}
showing that $(X,Z)$ in \eqref{4.14} is a trigonometric Calogero-Moser pair.

On the other hand, denoting for short by $\exp\{\ldots\}$ the expression that
appears inside the exponential in \eqref{4.13}, one computes
\begin{align*}
&\tau_V\Big(t_1-\frac{1}{z},t_2-\frac{1}{2z^2},t_3-\frac{1}{3z^3},\ldots\Big)\\
&=\mbox{det}
\Big\{I-X\exp\{\ldots\}\big(zI-(Z-I)\big)\big(zI-Z\big)^{-1}\Big\}\\
&=\mbox{det}\Big\{I-X\exp\{\ldots\}-X\exp\{\ldots\}\big(zI-Z\big)^{-1}\Big\},
\end{align*}
from which it follows that
\begin{multline*}
\frac{\tau_V\big(t_1-\frac{1}{z},t_2-\frac{1}{2z^2},t_3-\frac{1}{3z^3},
\ldots\big)}
{\tau_V(t_1,t_2,t_3,\ldots)}\\
=\mbox{det}\Big\{I-X\big(\exp^{-1}\{\ldots\}
-X\big)^{-1}\big(zI-Z\big)^{-1}\Big\}.
\end{multline*}
Putting $t_1=x, t_2=t_3=\ldots=0$ in this formula, shows that $\psi_V(x,z)$
in \eqref{4.15} satisfies Sato's formula, with $\tau_V$ as in \eqref{4.13}.
Since this formula determines the tau-function up to a constant, the proof is
complete.
\end{proof}
%%%%%%%%%%%%%%%%%%%%%%%%%%%%%%%%%%%%%%%%%%%%%%%%%%%%%%%%%%%%%%%%%%%%%%%%%%%%%%%%%%%%%%%%%%%%%%%%%%%%%% Bibliography
\begin{ack} \emph{The research of L.H. is supported by the Belgian
Interuniversity Attraction Pole P06/02 and the European Science Foundation
Program MISGAM.  E.H. acknowledges the support by grant MI 1504/2005 of the
National Fund "Scientific research" of the Bulgarian Ministry of Education and
Science.}
\end{ack}

\end{document}